\documentclass[psamsfonts,11pt]{amsart}
\usepackage{mathpazo}
\usepackage[top=3cm,bottom=2cm,left=3cm,right=3cm]{geometry}
\usepackage{overpic}
\usepackage{amsmath,amsthm,amssymb,amscd,amsfonts,amsbsy}
\usepackage{color}
\usepackage{hyperref,url}
\usepackage{float}
\usepackage{hyperref}
\usepackage{enumerate}
\usepackage{mathrsfs}  
\usepackage{mathrsfs}  
\usepackage{amsthm}
\usepackage{amsmath}
\usepackage{amsfonts}
\usepackage{amssymb}
\usepackage{amsthm}
\usepackage{amsmath}
\usepackage{amsfonts}
\usepackage{amssymb}
\usepackage{mathtools}
\usepackage{mathrsfs}
\usepackage{comment}
\usepackage{svg}

\usepackage{amsmath, amssymb, graphics, setspace}
\usepackage[utf8]{inputenc}
\usepackage{listings,xcolor}
\usepackage[english]{babel}

\newcommand{\mathsym}[1]{{}}
\newcommand{\unicode}[1]{{}}

\usepackage[normalem]{ulem}

\newcommand{\R}{\ensuremath{\mathbb{R}}}

\newcommand{\CO}{\ensuremath{\mathcal{O}}}

\newcommand{\CH}{\ensuremath{\mathcal{H}}}
\newcommand{\ov}{\overline}
\newcommand{\p}{\partial}
\newcommand{\dive}{\textrm{div\,}}

\newcommand{\de}{\delta}
\newcommand{\e}{\varepsilon}

\newcommand{\sgn}{\mathrm{sign}}

\usepackage[toc,page]{appendix}
\usepackage{tikz-cd}

\newtheorem {proposition} {Proposition}

\newtheorem {lemma} {Lemma}

\newtheorem {remark}{Remark}

\newtheorem {mtheorem} {Theorem}

\def\R{\mathbb R}

\title[Crossing limit cycles in piecewise polynomial vector fields]
{Growth estimate for the number of crossing limit cycles\\ in planar piecewise polynomial vector fields}
\author[Luana Ascoli and D. D. Novaes]
{Luana Ascoli and Douglas D. Novaes}

\address{Departamento de Matem\'{a}tica - Instituto de Matem\'{a}tica, Estat\'{i}stica e Computa\c{c}\~{a}o Cient\'{i}fica (IMECC) - Universidade
Estadual de Campinas (UNICAMP), \ Rua S\'{e}rgio Buarque de Holanda, 651, Cidade Universit\'{a}ria Zeferino Vaz, 13083-859, Campinas, SP,
Brazil}
\email{ddnovaes@unicamp.br}
\email{luanaascoli@ime.unicamp.br}

\usepackage{mathpazo}

\allowdisplaybreaks
\begin{document}



\maketitle
\begin{abstract} 
Motivated by the classical Hilbert's Sixteenth Problem, we extend some main developments obtained for Hilbert's number in the polynomial setting to the piecewise polynomial context. Specifically, we study the growth of the maximum number of crossing limit cycles in planar piecewise polynomial vector fields of degree $n$, denoted by $\CH_c(n)$. The best previously known general lower bound is $\CH_c(n)\geq 2n - 1$. In this work, we show that $\CH_c(n)$ grows at least as fast as $n^2/4.$ Furthermore, we prove that $\CH_c(n)$ is strictly increasing whenever it is finite, and that in such cases this maximum can be realized by piecewise polynomial systems whose crossing limit cycles are all hyperbolic. Finally, for the more restrictive class of piecewise polynomial Hamiltonian vector fields, we adapt the recursive construction of Christopher and Lloyd to demonstrate that the corresponding maximal number of crossing limit cycles, denoted by $\widehat{\CH}_c(n)$, grows at least as fast as $n\log n/(2\log 2)$, thereby improving previously established linear growth estimate.
\end{abstract}

	\section{Introduction}
Since its formulation in 1900, Hilbert’s Sixteenth Problem has played a key role in driving advances in the theory of planar vector fields. It concerns bounding the amount of limit cycles that may arise in planar polynomial vector fields of a given degree. When such a bound exists for a degree $n$, it is referred to as the $n$th Hilbert number $\CH(n)$. More specifically, the $n$th Hilbert number is defined by
\[
\CH(n):=\sup\{\pi(P,Q): \deg(P)\le n,\ \deg(Q)\le n\},
\]
with $\pi(P,Q)$ being the number of limit cycles of the polynomial vector field $X=(P,Q)$. Accordingly, the problem reduces to showing that $\CH(n)$ is finite for each $n\in\mathbb{N}$.

Except for the linear case ($n=1$), where limit cycles cannot exist, determining the finiteness of $\CH(n)$ has proven to be a highly challenging problem. Indeed, the question whether $\CH(n)$ is finite remains open already for $n = 2$. The main existing results concern lower bounds for the Hilbert number. For instance,  $\CH(2)\geq 4$ \cite{CW79,1980A}, $\CH(3)\geq 13$ \cite{LLY09}, and $\CH(4)\geq 28$ \cite{PT19}. Lower bounds for several other small values of $n$ are reported in \cite{PT19}. The standard approach to improve these bounds, for a fixed degree $n$, is to construct explicit examples exhibiting more limit cycles than previously known. For arbitrary values of $n$, the challenge is to generate limit cycles within families of planar polynomial vector fields in which the degree is treated as a parameter. In the early stages of research, a quadratic lower bound for the growth of $\CH(n)$ was established (see Otrokov \cite{O54}, Il’yashenko \cite{I91}, and Basarab-Horwath and Lloyd \cite{BL88}). Later, Christopher and Lloyd \cite{CL95} proved a lower bound of order $n^{2}\log n $ for the growth of $\CH(n)$, a result that has since been revisited and refined by several authors (see, for instance, \cite{HL12,IS00,LCC02,ACMP20}).

The key idea of the methodology developed by Christopher and Lloyd~\cite{CL95} is a recursive construction based on singular transformations of kind $(x,y)\mapsto(x^2-A,y^2-A)$. Starting from a planar polynomial vector field $X$ possessing a given number of limit cycles, this transformation produces a new vector field whose degree is doubled and which contains four diffeomorphic copies of $X$, one in each quadrant, and consequently four copies of all its limit cycles. Moreover, new centers created along the coordinate axes after the transformation give rise to additional limit cycles, whose number is proportional to the square of the degree of $X$. By iterating this process, they construct a sequence $(X_k)_{k\in\mathbb{N}}$ of polynomial vector fields of degree $2^k - 1$ and admitting at least  $S_{k+1} = 4S_k + (2^k - 1)^2$ limit cycles, which yields a lower bound for the growth of $\CH(n)$ of order $n^2\log n$.

The most recent advance in understanding the Hilbert number was achieved by Gasull and Santana~\cite{GS25}, who established the strict monotonicity of the function $\CH(n)$, provided it is finite. More precisely, they first showed that $\CH(n)$, if finite, can be realized by hyperbolic limit cycles, and subsequently that $\CH(n+1) \geq 1+\CH(n)$. Their approach begins with a polynomial vector field $X$ of degree $n$ having exactly $\CH(n)$ hyperbolic limit cycles. They then perform a singular reparametrization of time that preserves these limit cycles, increases the degree of $X$ by one, and creates a line of singularities. Next, a first linear perturbation is applied to destroy this line of singularities and generate a Hopf point. Finally, a second linear perturbation is introduced so that the system undergoes a Hopf bifurcation, giving rise to one additional hyperbolic limit cycle beyond the existing $\CH(n)$.

Over the past few decades, discontinuous differential systems have attracted increasing attention, driven both by their broad range of applications and by the mathematical challenges they pose. The second part of Hilbert's Sixteenth Problem has also been investigated within this framework (see, for instance, \cite{BCT24,CarFerNov2022b,CNT19,GouTor20,LL23,NOVAES2022133523}). In this paper, our main objective is to extend the aforementioned results to the case of crossing limit cycles in piecewise polynomial vector fields, that is, limit cycles that cross transversely the set of discontinuity. As expected, the strategies employed in the smooth case do not apply directly to this discontinuous context. Nevertheless, we are able to develop suitable adaptations of these methods to address the presence of discontinuities along the switching line, thereby establishing analogous results for crossing limit cycles of piecewise polynomial vector fields.

Let $ Z= (P, Q) $ be a piecewise polynomial planar vector field defined by  
\[
P(x,y) =
\begin{cases}
P^+(x,y), & x > 0,\\
P^-(x,y), & x < 0,
\end{cases}
\qquad
Q(x,y) =
\begin{cases}
Q^+(x,y), & x > 0,\\
Q^-(x,y), & x < 0,
\end{cases}
\]
where $ P^{\pm} $ and $ Q^{\pm} $ are polynomials. The trajectories of $Z$ is given by the Filippov convention \cite{Filippov88}. We define $\deg P:=\max\{\deg P^+,\deg P^-\}$ (analogously for $\deg Q$) and denote by $ \pi_c(P,Q) $ the number of crossing limit cycles of $ Z $.
The analogue of the Hilbert number in this discontinuous setting is defined as
\begin{equation*}\label{hilbert}
\CH_c(n) := \sup \bigl\{\, \pi_c(P,Q) : \deg P,\deg Q \le n \,\bigr\}.
\end{equation*}

As in the polynomial case, determining upper bounds for $\CH_c(n)$ remains a challenging problem, even for $n=1$, that is, for piecewise linear vector fields. The finiteness of $\CH_c(1)$ was established only in recent years by Carmona et al.\cite{CarFerNov2022b}, using an integral characterization of Poincaré half-maps \cite{CarmonaEtAl19} and fewnomial theory~\cite{kho}. We refer the interested reader to~\cite{CarmonaEtAl19b,CFSN-center,Carmona2022-mc}, where this technique has been further employed in other contexts. Most known results in this piecewise setting also provide lower bounds, such as $\CH_c(1)\ge 3$~\cite{HuanYang12,LNT15,LP12}, $\CH_c(2)\ge 12$~\cite{BCT24}, and $\CH_c(3)\ge 24$~\cite{GouTor20}, while the best general lower bound currently available is $\CH_c(n)\ge 2n-1$, established by Buzzi et al.~\cite{BFJ18}, yielding a linear lower estimate for the growth of $\CH_c(n)$. Although this lower bound is strictly increasing, this does not by itself imply that $\CH_c(n)$ is strictly increasing, a fact which, to date, has not been established.

Our first main result sharpens the known linear estimate by showing that the growth rate of  
$ \CH_c(n) $ is at least of order $ n^{2} $. In addition, we establish that $ \CH_c(n) $ is indeed strictly increasing, provided it is finite, as stated below:

\begin{mtheorem}\label{mainthmA}
$\CH_c(n)$ grows at least as fast as $n^{2}/4$. More precisely,
\begin{equation}\label{eq:liminf}
\liminf_{n\to\infty} \frac{\CH_c(n)}{n^{2}} \ge \frac{1}{4}.
\end{equation}
Furthermore, if $\CH_c(n) < \infty$ for some $n \in \mathbb{N}$, then $\CH_c(n+1) \ge 1+\CH_c(n).$
\end{mtheorem}

Theorem \ref{mainthmA} extend some of the main developments previously obtained for Hilbert's number in the polynomial setting to the piecewise polynomial context.

The number of crossing limit cycles has also been investigated for the more restrictive class of piecewise polynomial Hamiltonian systems, that is, under the assumption that the polynomial vector fields $(P^+,Q^+)$ and $(P^-,Q^-)$ are both Hamiltonian. In this context, we denote
\begin{equation*}\label{hilbert}
\widehat{\CH}_c(n) := \sup \bigl\{\pi_c(P,Q) : (P^{\pm},Q^{\pm}) \text{ are Hamiltonian v.f. and } \deg P,\deg Q \le n \bigr\}.
\end{equation*}

Although no general uniform upper bound has been established for the number of crossing limit cycles in this setting, the finiteness problem can be approached by means of algebraic equations provided by the Hamiltonians. For instance, in \cite{LL23} it is shown that there exist at most $[n^2/2]$ crossing limit cycles with only two pieces in piecewise polynomial Hamiltonian vector fields of degree $n$. For a general uniform upper bound, one must take into account the possible configurations of crossing limit cycles. Lower bounds have also been addressed in several works, see, for example, \cite{LL23,YANG20111026}. The best general lower bound depending on the degree of the piecewise polynomial Hamiltonian system known so far is $\widehat{\CH}_c(n) \ge n-1$. Our second main result strengthens this estimate by showing that the growth rate of $\widehat{\CH}_c(n)$ is at least of order $n \log n$, as stated below:

\begin{mtheorem}\label{mainthmB}
	$\widehat{\CH}_c(n)$ grows at least as fast as $(n\log n)/(2\log 2)$. More precisely,
	\begin{equation}\label{liminfHhat}
	\liminf_{n\to \infty} \frac{\widehat{\CH}_c(n)}{n\, \log \, n} \ge \frac{1}{2\log 2}.
	\end{equation}
\end{mtheorem}

The paper is structured as follows. In Section~\ref{sec:proofB}, we present the proof of Theorem~\ref{mainthmA}. The overall strategy follows the approach used in \cite{GS25} for the smooth case; however, several technical aspects must be adapted to accommodate the piecewise setting, for instance by incorporating the notion of pseudo-Hopf bifurcation. In Section~\ref{sec:proofA}, we provide the proof of Theorem~\ref{mainthmB}, where the Christopher–Lloyd methodology is suitably modified for the piecewise framework.

\section{Proof of Theorem \ref{mainthmA}}\label{sec:proofB}

We begin this section with a brief description of the pseudo-Hopf bifurcation (see \cite[Proposition~2.3]{CNT19}), originally reported by Filippov in his seminal book \cite{Filippov88} (see item~(b) on p.~241). This bifurcation is characterized by the emergence of a limit cycle when the sliding set induces a change in the stability of a monodromic equilibrium. Consider the one-parameter family of $\mathcal{C}^1$ piecewise vector fields
\[
Z_{b}(x,y)=
\begin{cases}
Z^+(x,y-b), & x>0,\\
Z^-(x,y), & x<0,
\end{cases}
\]
where $Z^{\pm}=(P^{\pm},Q^{\pm})$. Assume that, for $b=0$, the origin is a monodromic equilibrium of $Z_0$. For conciseness, we set $\ell=-1$ if the origin is attracting, $\ell=1$ if it is repelling, and $\ell=0$ otherwise. The parameter $\ell$ is related to the Lyapunov quantities. Moreover, let $a=1$ if the trajectories near the origin rotate counterclockwise, and $a=-1$ otherwise. If $\ell\neq 0$, then the vector field $Z_b$ undergoes a pseudo-Hopf bifurcation at $b=0$. More precisely, whenever $a\,b\,\ell<0$ and $b$ is sufficiently small, the vector field $Z_b$ admits a limit cycle surrounding a sliding segment (see Figure~\ref{fig:PH-bif}).
\begin{figure}[H]
	\centering
	\begin{overpic}[scale=0.6]{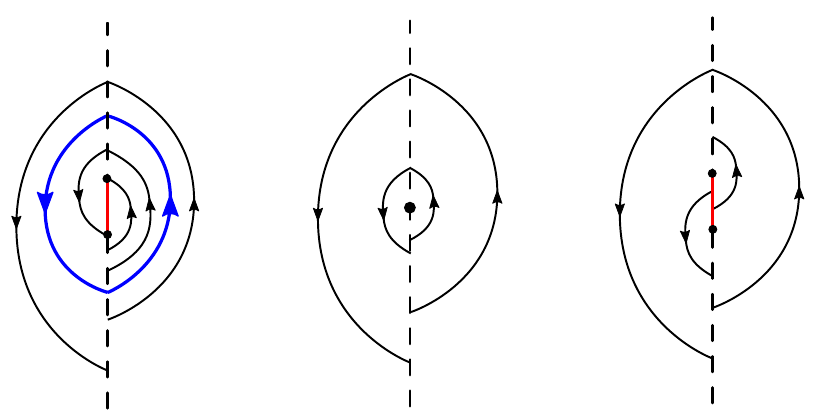}
			\put(10,-3){$b<0$}
			\put(47,-3){$b=0$}
			\put(84,-3){$b>0$}
	\end{overpic}
	\vspace{0.5cm}
	\caption{Pseudo-Hopf bifurcation in the case $\ell=a=1$.}
	\label{fig:PH-bif}
\end{figure}

The proof of Theorem \ref{mainthmA} is split in the next three propositions.
\begin{proposition}\label{prop:n2}
There exists a sequence of piecewise polynomial vector fields $(Z_k)_{k\in\mathbb{N}}$ with degree $n_k = 2^{k}-1$ and having $(n_k^2-1)/4$ crossing limit cycles.
\end{proposition}
\begin{proof}
Christopher and Lloyd in \cite{CL95}  (see also \cite[Section~3]{LCC02}) constructed a sequence of polynomial vector fields $(Z_k)_{k\in\mathbb{N}}$ of degree $n_k = 2^{k}-1$, each possessing at least $(2^{k-1}-1)^2$ hyperbolic limit cycles crossing the $y$-axis. More precisely, $Z_k$ has $2^{k-1}-1$ foci distributed along the $y$-axis, each surrounded by $2^{k-1}-1$ small-amplitude hyperbolic limit cycles. Moreover, $Z_k$ has the following form
\[
Z_k(x,y)
=
\Big(
X_k (y)
 + \e \big(P_{k-1}(x^2-a_k,y^2-b_k)+P_k(x)\big),
Y_k(x)
\Big),
\]
where $(X_k(y),Y_k(x))=( -\partial_y H_k(x,y), \p_x H_k(x,y) )$, $H_k$, $P_{k-1}$, and $P_k$ are polynomials, $a_k$ and $b_k$ are suitably chosen parameters, and $\e\neq 0$ is a small parameter. Computing the divergence of $Z_k$ along the $y$-axis yields
\[
\dive Z_k(0,y)=\e\, P_k'(0).
\]
Hence, each of the $2^{k-1}-1$ foci is hyperbolic, and all share the same stability. Moreover, the independence of the second coordinate of $Z_k(x,y)$ from the variable $y$ ensures that the orbits near all foci on the $y$-axis rotate in the same direction.

Now, for each $k$, consider the one-parameter family of piecewise polynomial vector fields
\[
Z_k^b(x,y)=
\begin{cases}
Z_k(x,y-b), & x>0,\\[2mm]
Z_k(x,y),   & x\le 0.
\end{cases}
\]
Since, for $b=0$, all foci of $Z_k^0 = Z_k$ on the $y$-axis have the same stability and the nearby orbits rotate in the same direction, they undergo a simultaneous pseudo-Hopf bifurcation as the parameter $b$ crosses zero. This bifurcation creates additional $2^{k-1}-1$ crossing limit cycles, each surrounding a sliding segment arising from the foci on the $y$-axis, while the original $(2^{k-1}-1)^2$ hyperbolic ones persist for $b$ sufficiently small as crossing limit cycles. Consequently, there exists a sufficiently small $b^* \neq 0$ such that $Z_k^{b^*}(x,y)$ exhibits $2^{k-1}(2^{k-1}-1)=(n_k^{2}-1)/4$ crossing limit cycles.
\end{proof}

\begin{remark}
Notice that Proposition~\ref{prop:n2} provides the estimate $\CH_c(n_k)\geq (n_k^2-1)/4$, from which~\eqref{eq:liminf} directly follows. Nevertheless, it is worth emphasizing that the validity of~\eqref{eq:liminf} does not essentially depend on the ``pseudo-Hopf bifurcation procedure'' employed in Proposition~\ref{prop:n2}. Indeed, the polynomial vector field $Z_k$ used in the proof of Proposition~\ref{prop:n2} may be regarded as a piecewise polynomial vector field separated by the $y$-axis, where the $(2^{k-1}-1)^2$ hyperbolic limit cycles crossing the $y$-axis are themselves crossing limit cycles. In this case, one obtains $\CH_c(n_k)\geq (2^{k-1}-1)^2=(n_k-1)^2/4$, which also provides the growth rate in~\eqref{eq:liminf}. The ``pseudo-Hopf bifurcation procedure'', in addition to increasing the number of limit cycles for each degree $n_k$, ensures the construction of a sequence of proper (discontinuous) piecewise polynomial vector fields.
\end{remark}

\begin{proposition}
Assume that $\CH_c(n)<\infty$ for some $n\in\mathbb{N}$. Then, there exists a piecewise polynomial vector field of degree $n$ possessing exactly $\CH_c(n)$ hyperbolic crossing limit cycles.
\end{proposition}
\begin{proof}
Let $Z(x,y)=(P(x,y),Q(x,y))$ be a piecewise polynomial vector field, where 
\[
P(x,y)=
\begin{cases}
P^+(x,y), & x>0,\\
P^-(x,y), & x\le 0,
\end{cases}
\qquad
Q(x,y)=
\begin{cases}
Q^+(x,y), & x>0,\\
Q^-(x,y), & x\le 0,
\end{cases}
\]
$P^\pm,Q^\pm$ are polynomials of degree $n$, and $Z$ has exactly $\CH_c(n)<\infty$ crossing limit cycles.  
Denote $Z^+=Z|_{x>0}$ and $Z^-=Z|_{x\le 0}$, and consider the perturbation of $Z$, for $b$ close to $0$,
\[
Z_b(x,y)=
\begin{cases}
Z^+(x,y-b), & x>0,\\
Z^-(x,y), & x\le 0.
\end{cases}
\]

Let $C$ be a crossing limit cycle of $Z_0 = Z$. Observe that at each point where $C$ intersects the discontinuity set, one can locally define a first return map, each possessing a distinct fixed point corresponding to that intersection. Let $\pi_C : S_C \to S_C$, $S_C \subset \Sigma,$
denote one of such maps and $y_C$ its fixed point. Accordingly, $\{(0, y_C)\} = C \cap S_C.$ Observe that $\pi_C$ can be written as the composition of finitely many half-return maps determined by $Z^+$ and $Z^-$, namely,
\[
\pi_C(y) = \phi_1^+ \circ \phi_1^- \circ \cdots \circ \phi_k^+ \circ \phi_k^-(y).
\]
Each map $\phi_i^+$ and $\phi_i^-$ corresponds to a half-return map locally defined by the flow of $Z^+$ and $Z^-$, respectively (see Figure~\ref{cycle}).
\begin{figure}[h]
	\includegraphics[scale=1]{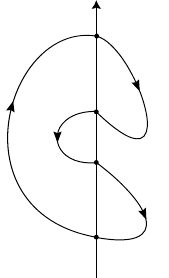}
	\put (-5,90) {\footnotesize$\phi_2^+$}
	\put (-5,20) {\footnotesize$\phi_1^+$}
	\put (-65,75) {\footnotesize$\phi_1^-$}
	\put (-80,30) {\footnotesize$\phi_2^-$}
	\put (-30,10) {\footnotesize$y$}
	\caption{Possible configuration of a limit cycle and the associate half-return maps.}\label{cycle}
\end{figure}
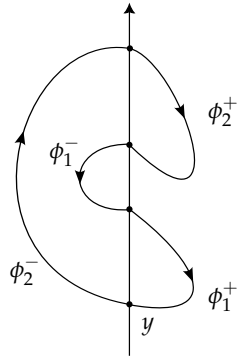
Moreover, since the crossing limit cycle is transverse to the discontinuity line, each of these maps is analytic and strictly decreasing. The analyticity follows from the analytic dependence of solutions on the initial condition, whereas their strictly decreasing character results from the uniqueness of solutions. In particular, $(\phi_i^{\pm})'(y) \leq 0.$ From transversality, the perturbed first return map remains well defined for $b \ne 0$ sufficiently small:
\[
\pi_{C}(y;b) = \phi_{1,b}^+ \circ \phi_1^- \circ \cdots \circ \phi_{k,b}^+ \circ \phi_k^-(y),
\]
where $\phi_{i,b}^+(y) = \phi_i^+(y - b) + b$, for $i = 1, \ldots, k.$ In addition,
\[
\begin{aligned}
\dfrac{\p \pi_C}{\p b}(y;b) =&1-(\phi_1^+)' +(\phi_1^+)' (\phi_1^-)'-(\phi_1^+)' (\phi_1^-)'(\phi_2^+)'+\cdots\\
& -\big[(\phi_1^+)' (\phi_1^-)'\cdots (\phi_{k-1}^+)' (\phi_{k-1}^-)'(\phi_k^+)'\big]+\big[(\phi_1^+)' (\phi_1^-)'\cdots (\phi_k^+)'(\phi_k^-)'\big]\geq 1.
\end{aligned}
\]
Thus, consider the displacement function associated with the crossing limit cycle $C$,
\[
\delta_C(y,b) := \pi_C(y;b)  - y.
\]
Observe that each isolated (resp. simple) zero of $\delta_C(\cdot,b)$ corresponds to a (resp. hyperbolic) crossing limit cycle of $Z_b$.
In particular, for $b=0$, the point $y_C$ is an isolated zero of $\delta_C(\cdot,0)$.
Since $\delta_C(\cdot, 0)$ is analytic, then it writes
\[
\delta_C(y,0) = a_C (y - y_C)^{\ell_C} + (y - y_C)^{\ell_C + 1} R_C(y),
\]
where $a_C \neq 0$, $\ell_C$ is a positive integer, and the function $R_C$ is analytic. 
Since,
\[
\dfrac{\partial \delta_C}{\partial b}(y,b)
= \dfrac{\partial \pi_b}{\partial b}(y)
\geq 1.
\]
the isolated zero $y_C$ unfolds differently in three distinct scenarios, namely:
\begin{itemize}
\item[(O)] $\ell_C$ is odd: one can find $\beta_C > 0$ for which the function $\delta_C(\cdot, b),$ for $b \in (-\beta_C, \beta_C)$, admits a curve of zeros $y_0 : (-\beta_C, \beta_C) \to \mathbb{R}$ satisfying $y_0(0) = y_C$. Moreover, for each $b \ne 0$, this zero is simple.

\item[($E^+$)] $\ell_C$ is even and $a_C > 0$: one can find  $\beta_C > 0$ for which $\delta_C(\cdot, b),$ for $b \in (-\beta_C, \beta_C)$, admits two  curves of zeros $y_1, y_2 : (-\beta_C, 0] \to \mathbb{R}$ satisfying $y_1(0) = y_2(0) = y_C$. Moreover, for each $b \ne 0$, these zeros are simple.

\item[($E^-$)]  $\ell_C$ is even and $a_C < 0$: one can find  $\beta_C > 0$ for which  $\delta_C(\cdot, b)$ admits two curves of zeros $y_1, y_2 : (0, \beta_C] \to \mathbb{R}$ satisfying $y_1(0) = y_2(0) = y_C$. Moreover, for $b \ne 0$, these zeros are simple.
\end{itemize}
Therefore, we say that $C$ is an $O$, $E^+$, or $E^-$ crossing limit cycle of $Z$ according to whether its associated fixed point $y_C$ is a zero of $\delta_C(\cdot,0)$ of type $O$, $E^+$, or $E^-$, respectively. Clearly, this characterization is independent of the choice of the first return map $\pi_C$. 

Finally, denote by $o$, $e^+$, and $e^-$ the numbers of $O$, $E^+$, and $E^-$ crossing limit cycles of $Z_0 = Z$, respectively. Notice that $o + e^+ + e^- = \CH_c(n).$ We may assume that $e^+ \ge e^-$. Taking the unfolding scenarios into account, we select the smallest value $\beta^*$ among the values $\beta_C$ corresponding to the $O$ and $E^+$ crossing limit cycles $C$ of $Z_0$. Hence, for each $b \in (-\beta^*,0)$, the function $\delta_C(\cdot,b)$ associated with each $O$ crossing limit cycle $C$ possesses a simple zero, while for each $E^+$ crossing limit cycle $C$, the function $\delta_C(\cdot,b)$ has two simple zeros. Consequently, the piecewise vector field $Z_b$ admits at least $o + 2e^+ \ge \CH_c(n)$ hyperbolic crossing limit cycles. Since, by hypothesis, $Z_b$ cannot have more than $\CH_c(n) < \infty$ crossing limit cycles, it follows that $Z_b$ has exactly $\CH_c(n)$ crossing limit cycles, all of them hyperbolic. 
\end{proof}

\begin{proposition}
Assume that $\CH_c(n)<\infty$ for some $n\in\mathbb{N}$. Then,  $\CH_c(n+1)\geq \CH_c(n)+1$.
\end{proposition}
\begin{proof}

Assume that $\CH_c(n)<\infty$. Let $Z(x,y)=(P(x,y),Q(x,y))$ be the piecewise polynomial vector field
\[
P(x,y)=
\begin{cases}
P^+(x,y), & x\geq 0,\\
P^-(x,y), & x\leq 0,
\end{cases}
\qquad
Q(x,y)=
\begin{cases}
Q^+(x,y), & x\geq 0,\\
Q^-(x,y), & x\leq 0,
\end{cases}
\]
where $P^\pm,Q^\pm$ are polynomials of degree $n$ and $Z$ has exactly $\CH_c(n)$ crossing limit cycles, al of them hyperbolic.

The following conditions (A) and (B) can be assumed without loss of generality:
\begin{itemize}
\item[(A)] All crossing limit cycles lie entirely within the open lower half-plane $\{(x,y):\, y<0\}$, and the origin is a crossing point, that is, $P^+(0,0)P^-(0,0)>0$.
\end{itemize}
Indeed, the crossing limit cycles intersect the discontinuity line at finitely many
points $(0,y_i)$, $i\in\{1,\ldots,k\}$, and all such intersections are transversal. Let
$y^*=\max_i y_i$. Then there exists $a>0$ such that every crossing limit cycle lies below
the curve $y=a|x|+y^*$. Otherwise, either the intersection of a limit cycle with the
discontinuity line at $(0,y^*)$ would fail to be transversal, or there would exist a limit
cycle intersecting the discontinuity line at a height strictly greater than $y^*$, both
of which are impossible. Consider now the piecewise linear homeomorphism $h(x,y)=(x,y-a|x|)$. Notice that $h(0,y)=(0,y)$, so $h$ fixes $\Sigma$ and also leaves the half-planes
$\{(x,y):x>0\}$ and $\{(x,y):x<0\}$ invariant. Moreover, for every $x\in\R$,
$h(x,a|x|+y^*)=(x,y^*)$, and therefore, under this transformation, all crossing limit
cycles are contained in the closed lower half-plane $\{(x,y):\, y\le y^*\}$ and remain
hyperbolic. In addition, the point $(0,y^*)$ is a crossing point, since it belongs to a
crossing limit cycle. Consequently, for $\hat y>y^*$ sufficiently close to $y^*$, the
point $(0,\hat y)$ is also a crossing point, and all crossing limit cycles lie entirely
in the open lower half-plane $\{(x,y):\, y<\hat y\}$. Finally, a translation sending
$(0,\hat y)$ to the origin ensures condition~(A).

\begin{itemize}
\item [(B)] The following nondegeneracy conditions hold,
\[
Q^+(0,0)Q^-(0,0)\neq 0,\quad \text{and} \quad P^+(0,0)Q^+(0,0)+P^{-}(0,0)Q^-(0,0)\neq0,
\]
\end{itemize}
It can be achieved by an arbitrarily small perturbation that preserves all hyperbolic crossing limit cycles.

Now, consider the following 1-parameter family of piecewise polynomial vector fields of degree $n+1$, $Z_{\e}(x,y)=(y P(x,y),Q_{\e}(x,y))$ where
\[
Q_{\e}(x,y)=
\begin{cases}
Q^+_{\e}(x,y)=\big(y-\e P^+(0,0)Q^+(0,0)\big)Q^+(x,y), & x>0,\\
Q^-_{\e}(x,y)=\big (y+\e P^-(0,0)Q^-(0,0)\big)Q^-(x,y), & x<0.
\end{cases}
\]
Notice that, for $\e=0$, $Z_0(x,y)=y Z(x,y)$ and, then, $Z_0|_{y<0}$ corresponds just to a time reparametrization of $Z|_{y<0}$. Taking condition $(A)$ into account, this implies that $Z_0|_{y<0}$ has exactly $\CH_c(n)$ hyperbolic crossing limit cycles and, therefore, for $\e>0$ sufficiently small $Z_{\e}$ has at least $\CH_c(n)$ hyperbolic crossing limit cycles contained the open lower half-plane $\{(x,y):\, y<0\}$.

We claim that, for $\e>0$, the origin is a monodromic two-fold singularity. Indeed, denoting by $Z_{\e}^+=Z_{\e}|_{x\geq 0}$ and $Z_{\e}^-=Z_{\e}|_{x\leq 0}$, and setting $h(x,y)=x$, it follows from conditions in (A) and (B) that
\[
\begin{aligned}
&Z_{\e}^+h(0,0)=Z_{\e}^-h(0,0)=0,\\
&(Z_{\e}^+)^2 h(0,0)=-\e (P^+(0,0))^2(Q^+(0,0))^2<0,\\
&(Z_{\e}^-)^2 h(0,0)=\e (P^-(0,0))^2(Q^-(0,0))^2>0.
\end{aligned}
\]
Thus, the origin is invisible-invisible two-fold singularity. In addition, 
\[
Q^+_{\e}(0,0)Q^-_{\e}(0,0)=-\e^2 P^+(0,0)P^-(0,0)(Q^+(0,0))^2(Q^-(0,0))^2<0,
\]
which guarantees the monodronicity around the origin. Now, following \cite{NS21} (see also \cite{gassulcoll}), we compute the second Lypunov coefficient that controls the stability of the origin as
\[
V_2(\e)=\dfrac{2\delta}{3\e}\dfrac{P^+(0,0)Q^+(0,0)+P^{-}(0,0)Q^-(0,0)}{P^+(0,0)Q^+(0,0)P^{-}(0,0)Q^-(0,0)}+\CO(1),
\]
where $\delta=\sgn(P^+(0,0))$. Notice that, under condition $(B)$, $V_2(\e)\neq 0$ for $\e>0$ sufficiently small. Thus, the origin is attracting (resp. repelling), provided that $V_2(\e)<0$ (resp. $V_2(\e)>0$). Thus, fix $\e=\ov{\e}>0$ in order that $Z_{\ov{\e}}$ has $\CH_c(n)$ hyperbolic crossing limit cycles and $V_2(\ov{\e})\neq 0$. Thus, in light of the characterization of the pseudo-Hopf bifurcation in the beginning of Section \ref{sec:proofB}, we have $\ell=\sgn(V_2(\ov{\e}))$ and $a=-\delta$.

Finally, consider the perturbation
\[
\widetilde Z_b(x,y)=
\begin{cases}
Z_{\ov{\e}}^+(x,y-b), & x>0,\\
Z_{\ov{\e}}^-(x,y), & x\le 0.
\end{cases}
\]
For $b=0$ we have $\widetilde Z_0 = Z_{\overline{\varepsilon}}$, which possesses 
$\CH_c(n)$ hyperbolic crossing limit cycles contained in the open lower half-plane 
$\{(x,y):\, y<0\}$. By hyperbolicity, these cycles persist for all sufficiently small 
$|b|$. Hence, $\widetilde Z_b$ has at least $\CH_c(n)$ hyperbolic crossing limit cycles 
lying in $\{(x,y):\, y<0\}$ whenever $|b|$ is small.  Moreover, $\widetilde Z_b$ undergoes a pseudo-Hopf bifurcation at $b=0$ (see, for instance, \cite[Proposition 4]{NS25}), which creates a sliding segment and an additional hyperbolic crossing limit cycle surrounding 
$S$, whenever $\sgn(b)=\sgn(\delta V_2(\ov{\e}))$ (that is, $a\,b\,\ell<0$). Therefore, there exists $\overline{b}$ arbitrarily close to $0$ such that 
$\widetilde Z_{\overline{b}}$ has at least $1+\CH_c(n)$ hyperbolic crossing limit cycles. 
Since $\widetilde Z_{\overline{b}}$ s a piecewise polynomial vector field of 
 degree $n+1$, 
we conclude that $\CH_c(n+1) \ge 1 + \CH_c(n).$
\end{proof}

\section{Proof of Theorem \ref{mainthmB}}\label{sec:proofA}
This section is devoted to the proof of Theorem~\ref{mainthmB}. Our strategy is to construct a sequence of piecewise polynomial Hamiltonian vector fields $(Z_k)_{k\in\mathbb{N}}$ of degree $n_k = 3\cdot 2^{k} - 1,$
each possessing $c_k = 3\cdot k\cdot 2^{k-1} + 1$ crossing limit cycles. Once such a sequence is obtained, Theorem~\ref{mainthmB} follows directly. Indeed, since  $k = \log((n_k+1)/3)/\log 2,$
we have  
\[
	\widehat{\CH}_c(n_k) \geq c_k = 1 + \frac{(n_k+1)\log\!\left(\frac{n_k+1}{3}\right)}{2\log 2},
\]
which yields the asymptotic estimate \eqref{liminfHhat}.
Thus, the remainder of this section will be devoted to the proof of the following result.

\begin{proposition}\label{prop:aux}
There exists a sequence $(Z_k)_{k\in\mathbb{N}}$  of piecewise polynomial Hamiltonian vector fields with degree $n_k=3\cdot 2^{k}-1$ and having $c_k=3\cdot k \cdot 2^{k-1}+1$   crossing limit cycles.
\end{proposition}

\subsection{Construction of the sequence of vector fields} The sequence of vector fields of Proposition~\ref{prop:aux} will be constructed inductively. Lemma~\ref{lemaaux1} below provides the first step of the induction, while Lemma~\ref{lemaaux2} establishes the inductive step.

Consider the perturbed polynomial Hamiltonian function of degree $d_0=3$,
\begin{equation}\label{eq:h0}
	H^{\pm}_{0}(x,y, \e) = 
	\frac{(\pm x-1)^2-y^2}{2}+\e P_0 ^{\pm}(y), 
\end{equation}
where $P_0 ^{\pm} =  a_{0,0} ^{\pm} + a_{0,1} ^{\pm}y +  a_{0,2} ^{\pm}y^2 + a_{0,3}^{\pm}y^3$ and $\e>0$ is a small parameter. The corresponding piecewise polynomial Hamiltonian vector field of degree $n_0=2$ is defined by
\begin{equation}\label{eq: z0}
	Z_{0,\e} = \left\{\begin{array}{ll}
		Z^+_{0,\e}\ \mbox{,}\ \mbox{if}\ x>0\\
		Z^-_{0,\e}\ \mbox{,}\ \mbox{if}\ x<0,
	\end{array}\hspace{0.5cm} \mbox{where}\right. \hspace{0.5 cm} Z^{\pm}_{0,\e} = \left(\frac{\partial}{\partial y} H_0^{\pm} (x,y,\e), -\frac{\partial}{\partial x} H_0^{\pm}(x,y,\e) \right).
\end{equation}

\begin{lemma}\label{lemaaux1}
	There exists coefficients of $P_0$ and $\bar{\e}>0$ such that $Z_{0,\e}$ has a hyperbolic crossing limit cycle contained in $\mathbb{R}\times(-2,2)$ for every $\e\in(-\bar{\e},\bar{\e})$.
\end{lemma}

\begin{proof}
Since $Z^{+}_{0,\e}$ and $Z^{-}_{0,\e}$ are Hamiltonian vector fields, their orbits lie in the level sets of the Hamiltonians $H^{+}_0$ and $H^{-}_0$, respectively. Thus,  define  $F_0^{\pm} : [0,1]\times[-1,0]\times[-l,l] \rightarrow \mathbb{R}$ by  
\begin{equation*}
\begin{aligned}
F_0^{\pm}(y_1, y_2, \e ) 
&= \frac{H^{\pm}_0 (0, y_1, \e) - H^{\pm}_0 (0, y_2, \e)}{y_1 - y_2}\\
&= -\frac{y_1 + y_2}{2} + \e\!\left(a^{\pm}_{0,1} + a^{\pm}_{0,2}(y_1 + y_2) + a^{\pm}_{0,3}(y_1^2 + y_1y_2 + y_2^2)\right).
\end{aligned}
\end{equation*}

Observe that, for each $y \in [0,1]$, the point $(y, -y, 0)$ satisfies  
\[
F_0^{\pm}(y, -y, 0) = 0 
\quad \text{and} \quad 
\frac{\partial F_0^{\pm}}{\partial y_2}(y, -y, 0) = -\frac{1}{2} \neq 0.
\]
Hence, by the Implicit Function Theorem and the compactness of $[0,1]$, there exist $\tilde{\e},\eta > 0$ small and a unique function  
\[
y_1^{\pm}: [0,1]\times(-\tilde{\e},\tilde{\e}) \rightarrow (-2,\eta)
\]
such that  
\begin{equation*} \label{eq:tf1}
y_1^{\pm}(y,0) = -y,
\quad\text{and}\quad
F_0^{\pm}\!\left(y, y_1^{\pm}(y, \e), \e\right) = 0,
\quad \forall (y,\e)\in [0,1]\times(-\tilde{\e},\tilde{\e}).
\end{equation*}
Moreover,
\[
\dfrac{\partial y_1^{\pm}}{\partial \e}(y, 0) = 2\!\left(a^{\pm}_{0,1} + a^{\pm}_{0,3}y^2\right), \quad \forall\, y\in[0,1].
\]
Note that the points $(0, y_1^{-}(y,\e))$ and $(0, y_1^{+}(y,\e))$ are, respectively, the return points on the $y$-axis of the trajectories of $Z_{0,\e}^{-}$ and $Z_{0,\e}^{+}$ passing through $(0,y)$ in forward and backward time.

Now, define the \emph{displacement function}
\[
\de_{0} : (0,1)\times(-\tilde{\e},\tilde{\e}) \rightarrow \mathbb{R}, \quad
\de_0(y, \e) := y_1^{+}(y,\e) - y_1^{-}(y,\e),
\]
whose zeros correspond to periodic orbits of $Z_{0,\e}$. Furthermore, isolated (resp. simple) zeros of $\de_0$ correspond to (resp. hyperbolic) crossing limit cycles of $Z_{0,\e}$. 
 Notice that
\begin{equation}\label{eq:disp0}
\de_0(y, \e)= \e\,\mathcal{M}_0(y) + \mathcal{O}(\e^2),\quad\text{where}\quad \mathcal{M}_0(y) =2(A_{0,1} + A_{0,3}y^2),
\end{equation}
and $A_{0,i} = a^{+}_{0,i} - a^{-}_{0,i}$ for $i = 1,3$. Thus, consider $\Delta_0 : [0, 1
   	]\times(-\tilde{\e},\tilde{\e}) \rightarrow \mathbb{R}$ defined by:
\[
\Delta_0(y, \e) := \frac{\de(y,\e)}{\e} = \mathcal{M}_0(y)+ \mathcal{O}(\e).
\]
There are coefficients $a_{0,i}^{\pm}$, $i=0,\ldots,3$, for which $\mathcal{M}_0$ has a unique simple zero $y_0\in (0,1)$. For these coefficients,
   	\begin{equation}\label{eq: A0}
   		\Delta_0(y_0, 0) =  0\ \ \mbox{and}\ \ \frac{\partial\Delta_0
   		}{\partial y} (y_0,0)\neq 0,
   	\end{equation}
and, therefor, the Implicit Function Theorem provides the existence of $\bar{\e}\in(0,\tilde{\e})$ and a unique function $y^*:(-\bar{\e},\bar{\e})\to (0,1)$ such that 
$y^*(0)=y_0$ and   $\Delta_0(y^*(\e), \e) = 0$ for all $\e\in(-\bar{\e},\bar{\e}).$
   	
Hence, the displacement function~\eqref{eq:disp0} admits a branch $y^*(\e)$ of simple zeros, and consequently $Z_{0,\e}$ possesses a hyperbolic crossing limit cycle for $\e \in (-\bar{\e}, \bar{\e})\setminus\{0\}$. Furthermore, $\bar{\e}>0$ can be chosen sufficiently small so that, since $y^{\pm}(y^*(\e), \e) > -2$, the corresponding limit cycle lies entirely within $\mathbb{R} \times (-2, 2)$.
\end{proof}

Now, starting from the polynomial Hamiltonian function $ H_0 $ of degree $ d_0 = 3 $ given in \eqref{eq:h0}, we recursively define a sequence of Hamiltonian functions $ (H_k)_{k \in \mathbb{N}} $. To this end, consider the singular transformation
\begin{equation*}\label{eq: Phi}
    \Phi : (x, y) \mapsto (x, y^2 - 2).
\end{equation*}
For each $ k \in \mathbb{N} $, we define a polynomial Hamiltonian function $ H_k^{\pm} $ of degree $ d_k = 2 d_{k-1} $ by
\begin{equation}\label{eq:Hk}
    H_k^{\pm}(x, y, \varepsilon, E_k) = H_{k-1}^{\pm} \bigl( \Phi(x, y), \varepsilon, E_{k-1} \bigr) + \varepsilon \, \varepsilon_k \, P_k^{\pm}(y),
\end{equation}
where $ E_k = (\varepsilon_1, \ldots, \varepsilon_k) \in \mathbb{R}^k $ and
\[
    P_k^{\pm}(y) = \sum_{j=0}^{d_k} a_{k,j} y^j.
\]

Finally, for each $k\in\mathbb{N}$, define the piecewise polynomial Hamiltonian vector field of degree $n_k:=d_k-1$ by
\begin{equation}\label{eq: zk}
		Z_{k,\e,E_k} = \left\{\begin{array}{ll}
		Z^+_{k,\e,E_k}\ \mbox{,}\ \mbox{if}\ x>0\\
		Z^-_{k,\e,E_k}\ \mbox{,}\ \mbox{if}\ x<0,
	\end{array}\right.\,\,\mbox{where} \,\, Z^{\pm}_{k,\e,E_k} = 
	\left(
	\frac{\partial H^{\pm}_k}{\partial y}(\cdot,  \e,E_k),
	-\,\frac{\partial H^{\pm}_k}{\partial x}(\cdot,\e,E_k)
	\right).
\end{equation}

It is important to mention that the transformation $\Phi(x,y) = (x, y^2 - 2)$ is not a diffeomorphism on the whole plane $\mathbb{R}^2$. Nevertheless, the restrictions $\Phi|_{\mathbb{R}\times (0,2)} : \mathbb{R}\times(0,2) \to \mathbb{R}\times(-2, 2)$ and $\Phi|_{\mathbb{R}\times (-2,0)} : \mathbb{R}\times(-2,0) \to \mathbb{R}\times(-2, 2)$ are diffeomorphisms. Consequently, the vector field $Z_{k+1,\e,E_{k+1}}|_{\e_{k+1}=0}$, when restricted to either $\mathbb{R}\times(0,2)$ or $\mathbb{R}\times(-2,0)$, is equivalent to $Z_{k,\e,E_k}$ restricted to $\mathbb{R}\times(-2,2)$, for all $k\in\mathbb{N}$ (see \autoref{fig:comportamentopelaphi}).

\begin{figure}[!htb]
	\centering
	\includegraphics[scale=0.45]{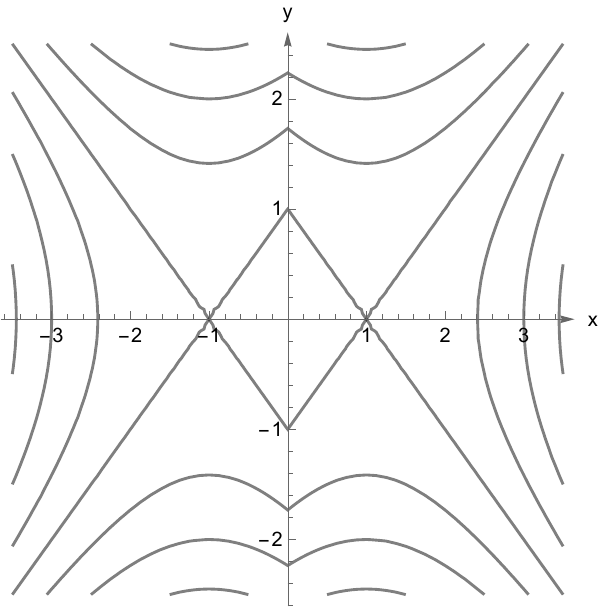}
	\includegraphics[scale=0.45]{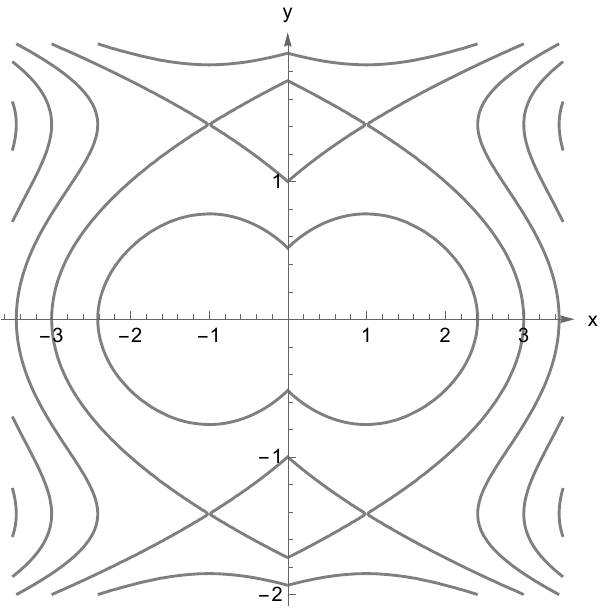}
	\includegraphics[scale=0.45]{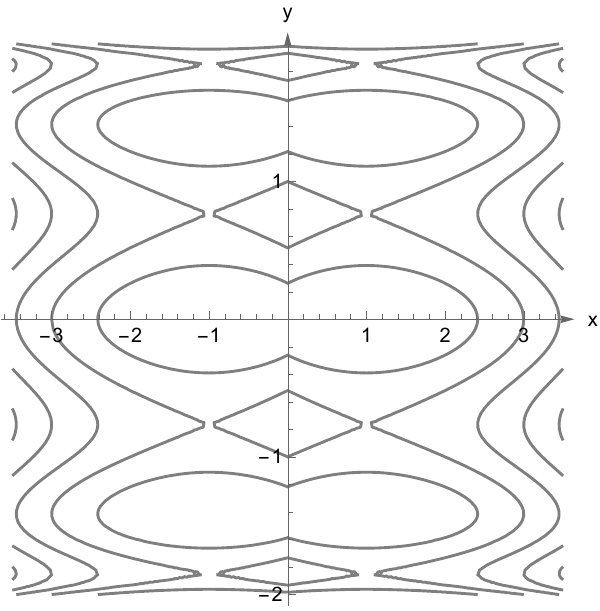}
	\caption{Level curves of $H_0$, $H_1$, and $H_2$, for $\e = 0$.}\label{fig:comportamentopelaphi}
\end{figure} 

Before proceeding to the next result, which ensures the prescribed number of crossing limit cycles for each $Z_{k,\e,E_k}$, we first establish the inductive hypothesis.

Assume that, for a fixed $E^*_k = (\e_1^*, \dots, \e_k^*) \in \mathbb{R}^k$ and every $\e \in (0, \hat{\e}]$, the vector field $Z_{k, \e, E^*_k}$ possesses $c_k$ hyperbolic crossing limit cycles contained in $\mathbb{R} \times (-2, \infty)$. For each $\e \in (0, \hat{\e}]$, denote by $(0, y_i(\e))$ and $(0, z_i(\e))$ the intersection points of the $i$-th limit cycle with the $y$-axis, where $y_i(\e) > z_i(\e)$ and $y_i(\e), z_i(\e) \neq 0$ for $i\in\{1, \dots, c_k\}$.  Define $ y_i(0) := \check{y}_i$ and assume that $\check{y}_i \neq 0$ for all $ i\in\{1, \dots, c_k\}$. This assumption ensures that each of these limit cycles converges, as $\e \to 0$, to a periodic orbit of $ Z_{k,0,E^*_k}$. Such a condition will be explicitly incorporated into our inductive hypothesis below. Accordingly, for each $ i\in\{1, \dots, c_k\}$, consider the displacement function defined in a neighborhood $ V_i$ of $\check{y}_i$, $\delta^i: V_i \rightarrow \mathbb{R}$, by
\begin{equation}\label{deltai}
	\delta^i(y,\e)= R^+_i(y,\e)-R^-_i(y,\e)
\end{equation} 
where $R_i^{\pm}(y, \e)$ denotes the half-return point to the $y$-axis, in forward or backward time respectively, of the orbit of $Z_{k,\e,E^*_k}$ passing through $(0, y)$. Hence, for all $\e \in [0, \hat{\e}]$, $\delta^i(y_i(\e), \e) = 0.$

We say that the $c_k$ hyperbolic crossing limit cycles of $Z_{k,\e,E_k^*}$ satisfy the hypothesis $(A_k)$ provided that
\[
\tag{$A_k'$}
	\check{y}_i \neq 0
	\quad \text{and} \quad
	\check{y}_i \neq \check{y}_j \ \text{for} \ i \neq j,
\]
and
\[
\tag{$A_k''$}
	\frac{\partial \delta^i}{\partial \e}(\check y_i, 0) = 0,
	\quad \text{and} \quad
	\frac{\partial^2 \delta^i}{\partial y \partial \e}(\check y_i, 0) \neq 0.
\]

Notice that condition $(A_k'')$ ensures that each one of these $c_k$ limit cycles converges to a distinct periodic orbit of $Z_{k,0,E^*_k}$. Moreover, condition $(A_k'')$ implies that, for each $i \in\{ 1, \dots, c_k\}$, $\check y_i$ is a simple zero of $\frac{\partial \delta^i}{\partial \e}(\cdot,0)$.

\begin{remark}\label{rem:firststep}
The relationship \eqref{eq: A0} implies that the hyperbolic limit cycle of $Z_{0, \e}$ provided by Lemma \ref{lemaaux1} satisfies hypothesis $(A_0)$.
\end{remark}

\begin{lemma} \label{lemaaux2}
Assume that there exists $E_k^*\in \mathbb{R}^{k}$ and $\hat{\e}>0$ such that,  for every $\e\in(0,\hat\e]$, the piecewise vector field $Z_{k,\e, E^*_k}$ has $c_k$ hyperbolic crossing limit cycles contained in $\mathbb{R}\times(-2,\ 2)$ and satisfying condition $(A_k)$. Then, there exist coefficients of $P^{\pm}_{k+1}$, $\bar{\e}\in(0,\hat{\e}],$ and $\bar{\e}_{k+1}> 0$ such that,  for every $\e\in (0,\bar{\e}]$ and $\e_{k+1} \in (0, \bar{\e}_{k+1})$, the piecewise vector field $Z_{k+1,\e, (E_k^*,\e_{k+1})}$ has $c_{k+1}= 2c_{k}+ d_k -1 $  hyperbolic crossing limit cycles contained in $\mathbb{R}\times(-2,\ 2)$ and satisfying condition $(A_{k+1})$.
\end{lemma}

\begin{proof}
Suppose that, for a fixed parameter vector  $E^*_k = (\varepsilon^*_{1}, \dots, \varepsilon^*_{k}) \in \mathbb{R}^k$ 
and for every $\varepsilon \in (0, \hat{\varepsilon})$,  the vector field $Z_{k, \varepsilon, E^*_k}$ possesses $c_k$ hyperbolic crossing limit cycles  contained in $\mathbb{R} \times (-2, \infty)$ and satisfying condition $(A_k)$.  Recall that, as $\varepsilon \to 0$, each of these limit cycles, whose largest intersection point with the $y$-axis is denoted by $(0, y_i(\varepsilon))$, converges to a distinct periodic orbit of $Z_{k, 0, E^*_k}$. In particular, $y_i(0) = \check{y}_i \neq 0$ and $\check{y}_i \neq \check{y}_j$ for $i \neq j$ (see condition $(A_k')$). Under the transformation $\Phi(x, y) = (x, y^2 - 2)$, this intersection corresponds to two points in the phase space of $Z_{k+1, \varepsilon, (E^*_k, 0)}$, namely 
$(0, -\tilde{y}_i)$ and $(0, \tilde{y}_i)$, where $\tilde{y}_i := \sqrt{\check{y}_i + 2}$. These points lie, respectively, on $\mathbb{R} \times (-2,0)$ and $\mathbb{R} \times (0,2)$. 

Let us analyze the point $(0, \tilde{y}_i)$.  The analysis for $(0, -\tilde{y}_i)$ is completely analogous.  
Notice that there exists a small neighborhood $\widetilde{V}_i \subset \mathbb{R}_+$ of $\tilde{y}_i$ such that, for $\varepsilon_{k+1}$ sufficiently small, the vector field $Z_{k+1,\varepsilon,(E_k^*,\varepsilon_{k+1})}$ admits a well-defined displacement function on $\widetilde{V}_i$. Moreover, taking \eqref{eq:Hk} into account, this displacement function can be written as
\begin{equation*}\label{tildedisp}
	\tilde{\delta}^i(y, \varepsilon, \varepsilon_{k+1}) 
	= \widetilde{R}^{+}_i(y, \varepsilon) - \widetilde{R}^{-}_i(y, \varepsilon) 
	+ \mathcal{O}(\varepsilon \, \varepsilon_{k+1}),
\end{equation*}
where $\widetilde{R}^{\pm}_i(y, \varepsilon)$ denote the return points on the $y$-axis of the orbits of 
$Z^{\pm}_{k+1, \varepsilon, (E^*_k, 0)}$ passing through $(0, y)$, for $y \in \widetilde{V}_i$.  
Moreover, $$\widetilde{R}^{\pm}_i(y, \varepsilon)=\sqrt{R^{\pm}_{i}(y^2 - 2, \varepsilon) + 2},$$ where $R_i^{\pm}(y, \varepsilon)$ are the half-return maps defined in a neighborhood of $\check{y}_i$ by the orbits of $Z_{k, \varepsilon, E^*_k}$ passing through $(0, y)$, as established in \eqref{deltai}. Hence,
\begin{equation*}\label{deltatil}
	\tilde{\delta}^i(y, \varepsilon, \varepsilon_{k+1}) 
	= \sqrt{R^{+}_{i}(y^2 - 2, \varepsilon) + 2} 
	- \sqrt{R^{-}_{i}(y^2 - 2, \varepsilon) + 2} 
	+ \mathcal{O}(\varepsilon \, \varepsilon_{k+1}).
\end{equation*}	
Thus, expanding $\tilde{\delta}^i$ at $\e=0$ and taking into account that $R_i^+(y,0)=R_i^-(y,0)$ for $y\in V_i$, we obtain
	\begin{equation*}
		\begin{aligned}
			\tilde{\delta}^i(y,\e, \e_{k+1})& =\ \tilde{\delta}^i(y, 0, \e_{k+1}) + \e \frac{\partial }{\partial \e}\tilde{\delta}^i(y, \e, \e_{k+1})\big|_{\e=0} +\mathcal{O}(\e^2)\\
			&=\e \left( \frac{\ \displaystyle{\frac{\partial\delta^{i}}{\partial\e}}(y^2-2,0)}{2\sqrt{R^{+}_{i}(y^2-2,0)+2}}\right)+ \mathcal{O}(\e\e_{k+1})+\mathcal{O}(\e^2).
		\end{aligned}
	\end{equation*}
	 Consider
	\begin{equation*}
		\tilde{\Delta}^{i}(y,\e, \e_{k+1}) := \frac{\tilde{\delta}^{i}(y, \e, \e_{k+1})}{\e} = \frac{\ \displaystyle{\frac{\partial\delta^{i}}{\partial\e}}(y^2-2,0)}{2\sqrt{R^{+}_{i}(y^2-2,0)+2}} + \mathcal{O}(\e_{k+1})+\mathcal{O}(\e).
	\end{equation*}
Note that, for each $i\in\{1,\ldots,c_k\}$,
	\begin{equation*}
		\vspace{-0.3cm}\tilde{\Delta}^{i}(\tilde{y}_i, 0, 0)= \frac{ \displaystyle{\frac{\partial\delta^{i}}{\partial\e}}(\tilde{y}_i^2-2,0)}{2\sqrt{R^{+}_{i}(\tilde{y}_i^2-2,0)+2}}\stackrel{\tilde{y}_i= \sqrt{\check y_i+2}}{=}\frac{ \overbrace{\frac{\partial^2\delta^{i}}{\partial y\partial\e}(y_i,0)}^{\quad\quad\ = 0\ (A_k)}}{2\sqrt{R^{+}_{i}(y_i,0)+2}} = 0,
	\end{equation*}
	and
	\begin{equation*}
		\begin{aligned}
			\frac{\partial\tilde{\Delta}^i}{\partial y}(\tilde{y}_i,0,0)&= \frac{\partial}{\partial y}\left(\frac{\ \displaystyle{\frac{\partial\delta^i}{\partial \e}}(\tilde{y}_i^2-2,0)}{2\sqrt{R^+_i(\tilde{y}_i^2-2,0)+2}}\right) \stackrel{\tilde{y}_i= \sqrt{\check y_i+2}}{=} \frac{\overbrace{\frac{\partial^2 \delta^i}{\partial y\partial\e} (y_i,0)}^{\quad \quad \neq 0\ (A_k)}\cdot\ 2\sqrt{R^{+}_i(y_i,0)+2}}{4(R^+_i(y_i,0)+2)} \neq 0. \\
		\end{aligned}
	\end{equation*}

Accordingly, since the set $\{1,\ldots,c_k\}$ is finite, the Implicit Function Theorem guarantees the existence of constants $\e^*\in(0,\hat \e]$ and $\e_{k+1}^*>0$ such that, for each $i\in\{1,\ldots,c_k\}$, there exists a unique function
\[
y_i^*:[-\e^*,\e^*]\times[-\e_{k+1}^*,\e_{k+1}^*]\longrightarrow (0,2)
\]
satisfying
\[
y_i^*(\e,0)=\tilde y_i
\quad\text{and}\quad
\widetilde{\Delta}^i\big(y_i^*(\e,\e_{k+1}),\e,\e_{k+1}\big)=0,
\]
for $(\e,\e_{k+1})\in[-\e^*,\e^*]\times[-\e_{k+1}^*,\e_{k+1}^*]$. Consequently, $Z_{k+1,\e,(E_k^*,\e_{k+1})}\big|_{\mathbb{R}\times(-2,2)}$ possesses at least $2c_k$ hyperbolic crossing limit cycles satisfying $(A_{k+1})$, for all $\e\in(0,\e^*]$ and $\e_{k+1}\in[0,\e_{k+1}^*]$. This concludes the first part of the proof.

In what follows, we investigate the bifurcation of additional limit cycles surrounding the origin. To this end, first observe that from \eqref{eq:Hk} we get
\begin{equation*}
    H^{\pm}_{k+1}(0,y,\e,E_{k+1})
    = H^{\pm}(0,\phi^{k+1}(y))
    + \e \Big(
        P^{\pm}_0(\phi^{k+1}(y))
        + \e_1\,P^{\pm}_1(\phi^{k}(y))
        + \cdots
        + \e_{k+1} P^{\pm}_{k+1}(y)
    \Big),
\end{equation*}
where $\phi(y) = y^{2} - 2$ denotes the second coordinate function of $\Phi$.
	
	Considering $\xi: \mathbb{R_+} \rightarrow \mathbb{R_+} $ such that $\xi(y) = 2 + \sqrt{y}$, define
	\begin{equation*}
		\begin{aligned}
			  I_1 = [0,1]&\mbox{,}\ \ J_1=[-1, 0],\\
			I_{k}=\left[0,\sqrt{2- \sqrt{\xi^{k-2}(3)}}\right]&\mbox{,}\ \ J_{k}=\left[-\sqrt{2-\sqrt{\xi^{k-2}(3)}}, 0\right]\mbox{,}\ \ k \in \mathbb{N}_{\geq 2}.
		\end{aligned}
	\end{equation*}
    Notice that, for $i=1,\hdots, k$, if $\ov{y}$ is a zero of $\phi^i$, then 
    \[
    \ov{y}= j_1\sqrt{2 + j_2\sqrt{2 + j_3\sqrt{\cdots + j_i\sqrt{2}}}}\,,\quad j_{\ell}\in\{-1,1\},\,\ell\in\{1,\ldots,i\}.
    \] 
     Hence, for all $y \in I_{k}\cup J_{k}$ and $i=1,\hdots, k$,\ $\phi^i(y) \neq 0$.
	
As in the previous lemma, we define  
\[
F_{k+1}^{\pm}: I_{k+1}\times J_{k+1}\times [-\e^*, \e^*] \times [0,\e_{k+1}^*]\rightarrow \mathbb{R}
\]
by
\begin{equation*}\label{f02}
	F_{k+1}^{\pm}(y_1, y_2, \e, \e_{k+1}) 
	= \frac{H^{\pm}_{k+1}(0, y_1, \e, (E^*_{k}, \e_{k+1})) 
	- H^{\pm}_{k+1}(0, y_2, \e, (E^*_{k}, \e_{k+1}))}{y_1 - y_2}.
\end{equation*}

Since $\phi$ is an even function, we have
\[
F^{\pm}_{k+1}(y,-y, 0,\e_{k+1})=0, 
\quad \forall\, y \in I_{k+1},\ \e_{k+1}\in[0,\e_{k+1}^*].
\]
Moreover, by induction on $k$,
\[
\frac{\partial F^{\pm}_{k+1}}{\partial y_2}(y,-y, 0,\e_{k+1})
= -2^k\prod_{i=1}^k \phi^{i}(y)\neq0,
\quad \forall\, y \in I_{k+1},\ \e_{k+1}\in[0,\e_{k+1}^*].
\]
Thus, since $I_{k+1}\times [0,\e_{k+1}^*]$ is compact, the Implicit Function Theorem guarantees the existence of $\tilde{\e}\in(0,\e^*]$, an interval $\tilde{J}_{k+1}$ such that
$J_{k+1}\subset \tilde{J}_{k+1} \subset (-2,\infty),$
and a unique function
\[
y_{1,k+1}^{\pm}: I_{k+1}\times[-\tilde{\e},\tilde{\e}]\times [0,\e_{k+1}^*]\rightarrow \tilde{J}_{k+1}
\]
satisfying
\begin{equation*}\label{eq:yk1}
	y_{1,k+1}^{\pm}(y,0,\e_{k+1})= -y
	\quad \text{and} \quad
	F_{k+1}^{\pm}\bigl(y, y_{1,k+1}^{\pm}(y,\e,\e_{k+1}), \e, \e_{k+1}\bigr)=0,
\end{equation*}
for $(y,\e,\e_{k+1})\in I_{k+1}\times[-\tilde{\e},\tilde{\e}]\times [0,\e_{k+1}^*]$. In addition,
\begin{equation}\label{ep:yqk}
	\frac{\partial y^{\pm}_{1,k+1}}{\partial \e}(y,0,\e_{k+1})
	= \e_{k+1} \frac{ P_{k+1}^{\pm}(y)-P_{k+1}^{\pm}(-y)}{2^{k+1} y\prod_{i=1}^{k+1}\phi^{i}(y)}
	= \e_{k+1}\frac{\sum_{j=1}^{d_k} a^{\pm}_{k+1,\,2j-1}y^{2j-2}}{2^k \prod_{i=1}^{k+1}\phi^{i}(y)},
\end{equation}
where, in the last equality, we have used that $d_{k+1}=2d_k$.

We now define the displacement function
\[
\delta_{k+1}: I_{k+1}\times[-\tilde{\e},\tilde{\e}]\times [0,\e_{k+1}^*]\rightarrow \mathbb{R}
\]
by
\begin{equation*}\label{eq:dk1}
	\delta_{k+1}(y,\e,\e_{k+1})
	:= y^+_{1,k+1}(y,\e,\e_{k+1})-y^-_{1,k+1}(y,\e,\e_{k+1}).
\end{equation*}
Then, from \eqref{ep:yqk}, it follows that
\[
\delta_{k+1}(y,\e,\e_{k+1})
= \e\,\mathcal{M}_{k+1}(y)+\mathcal{O}(\e^2),
\]
where
\[
\mathcal{M}_{k+1}(y)
= \e_{k+1}\frac{\sum_{j=1}^{d_k} A_j\,y^{2j-2}}{2^k \prod_{i=1}^{k+1}\phi^{i}(y)},
\qquad
A_j = a^+_{k+1,\,2j-1}-a^-_{k+1,\,2j-1}.
\]
Notice that we may choose the parameters $A_j$ so that $\mathcal{M}_{k+1}$ has $d_k-1$ simple zeros in $\operatorname{Int}(I_{k+1})$, denoted by $\{y_i \mid i=1,\ldots,d_{k-1}\}.$ Accordingly, we define
\[
\Delta_{k+1}: I_{k+1}\times[-\tilde{\e},\tilde{\e}]\times [0,\e_{k+1}^*]\rightarrow \mathbb{R}
\]
by
\[
\Delta_{k+1}(y,\e,\e_{k+1})
:= \frac{\delta_{k+1}(y,\e,\e_{k+1})}{\e}
= \mathcal{M}_{k+1}(y)+\mathcal{O}(\e).
\]
Since
\[
\Delta_{k+1}(y_i,0,\e_{k+1})=0
\quad \text{and} \quad
\frac{\partial}{\partial \e}\Delta_{k+1}(y_i,0,\e_{k+1})\neq 0,
\]
and since $\{1,\ldots, d_{k-1}\}$ is finite and $[0,\e_{k+1}^*]$ is compact, the Implicit Function Theorem implies the existence of $\bar{\e}\in(0,\tilde{\e})$ and, for each $i\in\{1,\ldots,d_{k-1}\}$, a unique function $y_i^*(\e,\e_{k+1})$ such that
\[
\Delta_{k+1}\bigl(y_i^*(\e,\e_{k+1}), \e, \e_{k+1}\bigr)=0,
\qquad
y_i^*(0,\e_{k+1})=y_i,
\]
for all $\e_{k+1}\in[0,\e_{k+1}^*]$ and $\e\in(0,\bar{\e})$.

Thus, in addition to the $2c_k$ crossing limit cycles that do not surround the origin, the vector field
$Z_{k+1,\e,E_{k+1}}\big|_{\mathbb{R}\times(-2,2)}$ admits $d_k-1$ hyperbolic crossing limit cycles surrounding the origin for all $\e \in (0,\bar{\e})$ and for and $E_{k+1} = (\e^*_1, \hdots, \e^*_{k}, \e_{k+1})$ with $\e_{k+1} \in (0, \e_{k+1}^*]$.
Moreover, these cycles satisfy condition $(A_{k+1})$.
\end{proof}

\subsection{Proof of Proposition \ref{prop:aux}\label{dempropaux}}
Consider the sequence $(Z_k)_{k\in\mathbb{N}}$ of piecewise polynomial Hamiltonian vector fields defined by
\[
Z_0 = Z_{0,\e} \quad \text{and} \quad Z_k = Z_{k,\e,E_k},
\]
as given in \eqref{eq: z0} and \eqref{eq: zk}, respectively.

Now, by Lemmas~\ref{lemaaux1} and~\ref{lemaaux2}, for each $k$ the parameter $\e$ and the parameter vector $E_k$ can be chosen so that $Z_0$ admits $c_0 = 1$ hyperbolic crossing limit cycle, whereas $Z_k$ admits $c_k = 2c_{k-1} + d_{k-1} - 1$ hyperbolic crossing limit cycles. Solving this recurrence relation, we obtain
\[
c_k = 3 \cdot k \cdot 2^{k-1} + 1.
\]

On the other hand, for each $k$ the vector field $Z_k$ has degree $n_k = 2d_{k-1} - 1,$
where $d_k = 2d_{k-1}$ with initial value $d_0 = 3$. Hence,
\[
n_k = 3 \cdot 2^{k} - 1.
\]

This completes the proof of Proposition~\ref{prop:aux}.

\begin{remark}
The estimate provided by Proposition \ref{prop:aux} can actually be sharpened. Indeed, performing a pseudo-Hopf bifurcation at each step (of Lemmas \ref{lemaaux1} and \ref{lemaaux2}) yields $c_0 = 2$ and $c_{k+1}= 2c_{k}+ d_k$, which gives $c_k = 3\cdot k\cdot 2^{k-1} + 2^k$. Nevertheless, this refinement does not lead to any improvement of the growth estimate for $\widehat\CH(n)$ in \eqref{liminfHhat}.
\end{remark}
	
	\section*{Acknowledgments}
	
DDN was supported by the São Paulo Research Foundation (FAPESP), Grant No. 2024/15612-6; by the Conselho Nacional de Desenvolvimento Científico e Tecnológico (CNPq), Grant No. 301878/2025-0; and by the Coordenação de Aperfeiçoamento de Pessoal de Nível Superior - Brasil (CAPES), through the MATH-AmSud program, Grant No. 88881.179491/2025-01. LA was supported by the São Paulo Research Foundation (FAPESP), Grant No. 2025/10287-2.
	
	\bibliographystyle{abbrv}
	
	\bibliography{references.bib}

\end{document}